\documentclass[conference]{IEEEtran}
\IEEEoverridecommandlockouts

\usepackage{cite}
\usepackage{amsmath,amssymb,amsfonts, amsthm}
\usepackage{graphicx}
\usepackage{textcomp}
\usepackage{xcolor}
\usepackage{amsmath, amssymb, amsthm, amsfonts}
\usepackage{mathtools}
\usepackage{bm}
\usepackage{comment}
\usepackage{hyperref}
\hypersetup{
    colorlinks=true,
    linkcolor=blue,
    filecolor=magenta,      
    urlcolor=blue,
    pdftitle={Overleaf Example},
    pdfpagemode=FullScreen,
    }
\usepackage{algorithm}
\usepackage[noend]{algpseudocode}
\usepackage{tabularx}
\usepackage{comment}

\newtheorem{define}{Definition}
\newtheorem{theorem}{Theorem}
\newtheorem{lemma}{Lemma}
\newtheorem{corr}{Corollary}
\newtheorem{remark}{Remark}
\newtheorem{assume}{Assumption}

\newtheorem{note}{Note}

\newcommand{\R}{\mathbb R}

\newcommand{\eps}{\epsilon}

\newcommand{\Z}{\mathbb{Z}}

\newcommand{\bmx}[1]{\begin{bmatrix}#1\end{bmatrix}} 
\newcommand{\bkt}[1]{\left[#1\right]} 
\newcommand{\pth}[1]{\left(#1\right)} 
\newcommand{\brc}[1]{\left \{#1\right \}} 
\newcommand{\nrm}[1]{\left \lVert#1\right \rVert} 

\newcommand{\bmxs}[1]{\begin{bsmallmatrix}#1\end{bsmallmatrix}}

\DeclarePairedDelimiter{\ceil}{\lceil}{\rceil}
\DeclarePairedDelimiter{\floor}{\lfloor}{\rfloor}
\DeclarePairedDelimiter{\abs}{\lvert}{\rvert}

\newcommand{\rarr}{\rightarrow} 

\makeatletter
\let\oldceil\ceil
\def\ceil{\@ifstar{\oldceil}{\oldceil*}}
\makeatother
\makeatletter
\let\oldfloor\floor
\def\floor{\@ifstar{\oldfloor}{\oldfloor*}}
\makeatother
\makeatletter
\let\oldnorm\norm
\def\norm{\@ifstar{\oldnorm}{\oldnorm*}}
\makeatother
\makeatletter
\let\oldabs\abs
\def\abs{\@ifstar{\oldabs}{\oldabs*}}
\makeatother

\newcommand{\revision}[1]{{\color{black}#1}}

\newcommand{\pdrv}[2]{\frac{\partial #1}{\partial #2}}

\newcommand{\cl}{\textup{cl}}

\newtheorem{problem}{Problem}

\algdef{SE}[SUBALG]{Indent}{EndIndent}{}{\algorithmicend\ }%
\algtext*{Indent}
\algtext*{EndIndent}

\makeatletter
\newcommand{\multiline}[1]{%
  \begin{tabularx}{\dimexpr\linewidth-\ALG@thistlm}[t]{@{}X@{}}
    #1
  \end{tabularx}
}
\makeatother
\def\BibTeX{{\rm B\kern-.05em{\sc i\kern-.025em b}\kern-.08em
    T\kern-.1667em\lower.7ex\hbox{E}\kern-.125emX}}

\IEEEoverridecommandlockouts

\begin{document}

\title{\vspace{0.3in} Safety for Time-Varying Parameterized Sets \\ Using Control Barrier Function Methods
}

\author{James Usevitch$^{1}$ and Jackson Sahleen$^{1}$
\thanks{$^{1}$Electrical and Computer Engineering Department, Brigham Young University, Provo, UT, USA
        {\{\tt\small james\_usevitch,jsahleen\}@byu.edu}}%
}

\maketitle

\begin{abstract}
A fundamental and classical problem in mobile autonomous systems is maintaining the safety of autonomous agents during deployment.
Prior literature has presented techniques using control barrier functions (CBFs) to achieve this goal.
These prior techniques utilize CBFs to keep an isolated point in state space away from the unsafe set.
However, various situations require a non-singleton set of states to be kept away from an unsafe set.
Prior literature has addressed this problem using nonsmooth CBF methods, but no prior work has solved this problem using only ``smooth" CBF methods.
This paper addresses this gap by presenting a novel method of applying CBF methods to non-singleton parameterized convex sets. The method ensures differentiability of the squared distance function between ego and obstacle sets by leveraging a form of the log-sum-exp function to form strictly convex, arbitrarily tight overapproximations of these sets. Safety-preserving control inputs can be computed via convex optimization formulations. The efficacy of our results is demonstrated through multi-agent simulations.
\end{abstract}

\section{Introduction}

Maintaining safety is a fundamental task of mobile autonomous systems.
A popular family of methods used in prior literature to accomplish this task is that of control barrier functions (CBFs) combined with online convex optimization to compute safety-preserving control  inputs.
These methods have mainly been used to keep a single point or state within a safe set and away from the opposing unsafe set.
CBF methods have proven to be powerful in a variety of settings including multi-agent collision avoidance \cite{cheng2020safe, tan2021distributed, wu2024constrained}, satellite trajectory design \cite{breeden2023robust}, safe dynamic bipedal walking \cite{nguyen2020dynamic}, coordinating accomplishment of signal temporal logic tasks \cite{lindemann2018control, lindemann2020barrier}, and applying runtime assurances to controllers based on machine / reinforcement learning \cite{van2024safe, marvi2021safe}.
A more complete overview of applications may be found in \cite{cohen2024safety,ames2019control,garg2024advances}.

There are situations, however, in which it is necessary to guarantee safety for an entire \emph{set of points} rather than a single point.
Examples include a robot with a polytope-shaped body that must avoid collisions with external obstacles, or a set of multiple agents collaboratively carrying a heavy object in a formation where the convex hull must avoid collisions. Maintaining safety in these scenarios may be accomplished in specific cases by defining an ellipsoid-shaped safe set CBF surrounding a single point, but these overapproximations may be too conservative for certain scenarios. This motivates the problem of generalizing CBF methods to \revision{guarantee safety for \emph{entire sets}} rather than single points in the state space.
This problem was first considered in \cite{thirugnanam2022duality, thirugnanam2022safety, thirugnanam2023nonsmooth}, which presented methods to guarantee safety and compute safety-preserving control inputs for convex ego sets. However, these methods rely on nonsmooth control barrier functions (NCBFs) which require special considerations when implementing in practice such as tracking the almost active constraints and a potentially increased number of QP constraints.

This paper considers the following question: Is it possible to apply CBF methods to guarantee safety of sets without resorting to NCBF techniques? Our results answer this question in the affirmative. More specifically, given a time-varying parameterized ego set, we present a novel method to compute control inputs that prevent the ego set from intersecting with an unsafe set. Our method does not require the use of NCBF methods, and control inputs can be computed using convex optimization formulations.
Our specific contributions in this paper are summarized as follows:
\begin{itemize}
    \item We present a novel method to compute safety-preserving control inputs for an ego set using convex optimization. Our method does not require the use of discontinuous systems theory and can be applied to a large class of compact sets defined by convex inequalities.
    \item We present a novel method to approximate the squared distance between convex sets in a differentiable form. The approximation can be made arbitrarily tight, has guaranteed unique optimal points, and is differentiable with respect to the set parameters.
\end{itemize}

A summary of the paper organization is as follows: in Section \ref{sec:notation_prob_formulation} we give the notation and problem formulation considered in this paper. Section \ref{sec:main_results} contains our main results on applying CBF methods to parameterized sets. Section \ref{sec:simulations} presents simulations demonstrating the efficacy of our method, and Section \ref{sec:conclusion} gives a brief conclusion.

\section{Notation and Problem Formulation}
\label{sec:notation_prob_formulation}


We adopt the following notation.
Vectors (e.g. \(\bm x \in \R^N\)) and functions having vector codomains, e.g. \(\bm f: \R^M \rarr \R^N\) are denoted with bold typing. Gradients of \(\bm f\) are denoted \(\pdrv{\bm f}{\bm x}\), and Hessians are denoted \(\pdrv{^2 \bm f}{\bm x^2}\).
The power set is denoted \(2^S\) for some set \(S\).
The closure of \(S\) is denoted \(\cl(S)\).
This paper will make extensive use of variants of the log-sum-exp function (LSE).
Let \(\bm x \in \R^q\). The LSE function is a convex smooth approximation of the maximum function and is defined as \(LSE(\bm x) \triangleq \log\pth{\sum_{i=1}^q e^{x_i}}\).
A strictly convex variant of the LSE function can be defined as \cite{nielsen2018monte} \(LSE^+(\bm x) \triangleq LSE(0,x_1,\ldots,x_q).\)
Lemma \ref{lem:LSE_positive_def} gives a proof of strict convexity.
We define the following strictly convex variant of the LSE function with parameter $\epsilon > 0$ as \cite{usevitch2022adversarial}\revision{:}
\begin{align*}
    LSE_\epsilon^+(\bm x) \triangleq \frac{1}{\epsilon} LSE^+(\epsilon \bm x)  
    = \frac{1}{\epsilon}LSE(0, \epsilon x_1, \ldots, \epsilon x_q).
\end{align*}

\revision{
\begin{lemma}
    \label{lem:LSE_bounds}
    The following holds for the \(LSE_\epsilon^+\) function:
    \begin{align}
        \max\pth{0,\bm x} &< LSE_\epsilon^+(\bm x) \leq \max\pth{0,\bm x} + \frac{\log(q)}{\epsilon}. \label{eq:LSE_plus_def}
    \end{align}
\end{lemma}

\begin{proof}
    See the Appendix for proofs of all Lemmas.
\end{proof}
}

\begin{lemma}
    \label{lem:LSE_positive_def}
    Let \(\eps > 0\). The function \(LSE_\epsilon^+\) has a positive definite (PD) Hessian and is strictly convex.
\end{lemma}

The following definition will be used \revision{for overapproximations of sets with a tightness} controlled by the value of \(\eps\).
\begin{define}
    \label{def:eps_plus}
    Let \(F^j(\bm x, \revision{\bm \lambda})\) be a set of functions parameterized by \(\revision{\bm \lambda \in \R^M}\) that are each convex in their first argument \(\bm x \in \R^N\). Let the set-valued map \(S : \R^M \to 2^{\R^M}\) be defined as \(S(\revision{\bm \lambda}) = \{\bm x \in \R^N: F^j(\bm x, \revision{\bm \lambda}) \leq 0\ \forall j = 1,\ldots,q \}\). Then the set-valued map \(S_\epsilon^+(\revision{\bm \lambda})\) for \(\eps > 0\) is defined as
    \begin{align*}
        S_\epsilon^+(\revision{\bm \lambda}) \triangleq  &\brc{\bm x: LSE_\epsilon^+(F^1(\bm x,\revision{\bm \lambda}),\ldots,F^q(\bm x, \revision{\bm \lambda})) \leq \frac{\log(q)}{\epsilon} }
    \end{align*}
\end{define}

\begin{lemma}
    \label{lem:S_subset_Sepsplus}
    Let \(S(\cdot),\ S_\epsilon^+(\cdot)\) be set-valued maps as defined in Definition \ref{def:eps_plus}. Then \(S(\revision{\bm \lambda}) \subseteq S_\epsilon^+(\revision{\bm \lambda})\) for all \(\bm \lambda \in \R^M\).
\end{lemma}

\subsection{Problem Formulation}
\label{sec:problem_formulation}

We consider the state space \(\bm x \in \R^N\) divided into a parametric safe set $S(\bm \lambda(t)) \subseteq \R^N$ and unsafe set \(\overline{S}(\bm \lambda(t)) = \R^N \backslash S(\bm \lambda(t))\) where \revision{\(\bm \lambda : \R \to \R^M\)} represents a possibly time-varying vector of parameters.
The dependence on \(t\) will be omitted for brevity when the context is clear.

\revision{
\begin{assume}
    \label{assume:unsafe}
    The unsafe set \(\overline{S}(\bm \lambda)\) is a subset of the union of a finite number of sets \(\overline{S}_j(\bm \lambda_j)\), i.e. \( \overline{S}(\bm \lambda) \subseteq \bigcup_{j=1}^J \cl(\overline{S}_j(\bm \lambda_j))\) for all \(\bm \lambda = \bmx{\bm \lambda_1^\intercal & \cdots & \bm \lambda_J^\intercal}^\intercal\)), where \(\bm \lambda_j \in \R^{M_j}\ \forall j \in 1 \ldots, J\).
\end{assume}
}

We also consider a convex and compact parameterized ego set $S_E(\revision{\bm \lambda_E}(t))$ where \(\revision{\bm \lambda_E}: \R \rarr \R^{M_E}\) is a time-varying vector with the following control affine dynamics:
\begin{align}
    \label{eq:parameter_dynamics}
    \revision{\dot{\bm \lambda}_E} = \bm f(\revision{\bm \lambda_E}) + \bm g(\revision{\bm \lambda_E}) \bm u.
\end{align}
The functions \revision{\(\bm f: \R^{N_E} \rarr \R^{N_E}\), \(\bm g: \R^{N_E} \rarr \R^{N_E \times P_E}\)} are assumed to be locally Lipschitz continuous in their argument.
The ego set's dependence on $\revision{\bm \lambda_E}(t)$ will be omitted for brevity when the context is clear.

\revision{
\begin{define}
    \label{def:standard_conditions}
    We say that a set-valued map \(S : \R^{M'} \to 2^{\R^{N}} \) satisfies the standard conditions if all of the following properties hold:
    \begin{enumerate}
        \item[1)] The set \(S(\bm \lambda')\) is convex, compact and has nonempty interior for all \(\bm \lambda' \in \R^{M'}\).
        \item[2)] There exist convex functions \(F^k: \R^N \times \R^{M'}\) parameterized by \(\bm \lambda'\) such that \(S(\bm \lambda') = \{\bm x: F^k(\bm x, \bm \lambda') \leq 0\}\), \(\forall k=1,\ldots,n_{F}\) where \(n_{F} \in \Z_{\geq 1}.\)
        \item[3)] The functions \(\frac{\partial F^k}{\partial x}\) are continuously differentiable in \(\bm \lambda'\) and twice continuously differentiable in \(\bm x\).
    \end{enumerate}
\end{define}

\begin{assume}
    \label{assume:applicable_sets}
    The sets $S_E$ and $\overline{S}_j\ \forall j \in \{1, \ldots, J\}$ satisfy the standard conditions stated in Definition \ref{def:standard_conditions}.
\end{assume}
}

\revision{For brevity, we define \(\bm F_E(\bm x, \bm \lambda_E) \triangleq \bmxs{F_E^1(\bm x, \bm \lambda_E) & \ldots & F_E^{n_{F_E}}(\bm x, \bm \lambda_E)}\) and \(\bm F_j(\bm x, \bm \lambda_j) \triangleq \bmxs{F_j^1(\bm x, \bm \lambda_j) & \ldots & F_j^{n_{F_j}}(\bm x, \bm \lambda_j)}\) as the convex functions from Definition \ref{def:standard_conditions} for the sets $S_E$ and $\overline{S}_j,\ j \in \{1, \ldots, J\}$ respectively.}
One final assumption is made on the gradients of the functions \revision{\(\bm F_j\)}. 

\revision{
\begin{assume}
    \label{assume:full_column_rank}
    The gradient matrices \(\pdrv{F_E(\bm x, \bm \lambda_E)}{\bm x}\) and \(\pdrv{\bm F_j(\bm x, \bm \lambda_j)}{\bm x}\) have full column rank for all \(\bm x \in \R^N,\ \bm \lambda_E \in \R^{M_E},\ \bm \lambda_j \in \R^{M_j}\).
\end{assume}
}

\begin{note}
    The conditions of Assumption \ref{assume:full_column_rank} are distinct from the well-known Linear Independence Constraint Qualification (LICQ) which stipulates that \revision{\(\pdrv{\bm F_E(\bm x, \bm \lambda_E)}{\bm x}\) and \(\pdrv{\bm F_j(\bm x, \bm \lambda_j)}{\bm x}\)} each have full \emph{\textbf{row}} rank.
\end{note}

The following Lemma demonstrates one practical case under which the conditions of Assumption \ref{assume:full_column_rank} hold.

\begin{lemma}
    \label{lemma:assume_full_column_rank}
    Suppose the functions \revision{\(F_E^i,\ F_j^k\) are affine and the sets \(\{\bm x : \bm F_E(\bm x, \bm \lambda_E) \leq 0 \}\) and \(\{\bm x: \bm F_j(\bm x, \bm \lambda_j) \leq 0 \}\) are compact polytopes with nonempty interiors.} Then the conditions of Assumption \ref{assume:full_column_rank} are satisfied.
\end{lemma}

The goal of this paper is to compute the control inputs $\bm u$ for the parameter dynamics \eqref{eq:parameter_dynamics} such that the ego set is a subset of the safe set for all forward time. More specifically, we consider the following problem:

\begin{problem}
    \label{prob:computecontrolinputs}
    Compute a control input $\bm u$ for the ego set that ensures $S_E(\revision{\bm \lambda_E}(t)) \subset S(\bm \lambda(t))$ for all $t \geq t_0$.
\end{problem}

\section{Main Results}
\label{sec:main_results}

The work in \cite{thirugnanam2023nonsmooth} proposed a method to prevent collisions between sets using control barrier function methods. Specifically, they proposed the minimum distance between sets as a candidate control barrier function. Consider the ego set \(S_E\) and one of the unsafe subsets \(\overline{S}_j\). The case handling all subsets \revision{\(\overline{S}_1,\ldots,\overline{S}_J\)} simultaneously will be considered in Section \ref{sec:computing_safe_input}. The squared minimum distance between \(S_E\) and \revision{\(\overline{S}_j\)} can be computed as
\begin{align}
    \begin{aligned}
            d(\revision{\bm \lambda_E}, \bm \lambda_j) &= \min_{\revision{\bm z_E}, \revision{\bm z_j}}&& \frac{1}{2}\nrm{\revision{\bm z_E} - \revision{\bm z_j}}_2^2 \\
            &\textup{s.t.} && \revision{\bm z_E} \in S_E(\revision{\bm \lambda_E}),\ \revision{\bm z_j} \in \overline{S}_j(\bm \lambda_j)
            \label{eq:min_dist_func}
    \end{aligned}
\end{align}
\revision{where $z_E, z_j$ represent points within the ego set and unsafe set respectively.} This optimization problem is convex. However, in general the optimal points \(\revision{\bm z_E^*},\revision{\bm z_j^*}\) of this formulation are not guaranteed to be unique. Notable examples include cases where \(S_E(\revision{\bm \lambda_E})\) and \(\overline{S}_j(\bm \lambda_j)\) are both polytopes. In addition, our use of CBF methods requires the function \(d\) to be differentiable with respect to the parameters \(\revision{\bm \lambda_E}, \bm \lambda_j\). Further assumptions on the sets \revision{\(S_E, \overline{S}_j\)} are required to guarantee this differentiability.

Rather than directly restrict the classes of convex sets that \(S_E(\revision{\bm \lambda_E}), \overline{S}_j(\bm \lambda_j)\) may take, we present an alternative reformulation of \eqref{eq:min_dist_func} where these sets are replaced with the overapproximations \((S_E)_\epsilon^+(\revision{\bm \lambda_E}), (\overline{S}_j)_\epsilon^+(\bm \lambda_j)\) using the \(LSE_\epsilon^+\) function:
\begin{align}
    \begin{aligned}
        d_\epsilon^+(\revision{\bm \lambda_E}, \bm \lambda_j) = &\min_{\revision{\bm z_E}, \revision{\bm z_j}}&& \frac{1}{2} \nrm{\revision{\bm z_E} - \revision{\bm z_j}}_2^2 \\
        &\hspace{-4em} \textup{s.t.} && \hspace{-4em} \revision{\bm z_E} \in (S_E)_\epsilon^+(\revision{\bm \lambda_E}),\ \revision{\bm z_j} \in (\overline{S}_j)_\epsilon^+(\bm \lambda_j)
        \label{eq:LSE_min_dist_func}
    \end{aligned}
\end{align}
Expanding the constraints yields the equivalent form,
\begin{align}
    d_\epsilon^+(\revision{\bm \lambda_E}, \bm \lambda_j) = &\min_{\revision{\bm z_E}, \revision{\bm z_j}} && \frac{1}{2}\nrm{\revision{\bm z_E} - \revision{\bm z_j}}_2^2 \nonumber \\
    &\hspace{-3em} \textup{s.t. } && \hspace{-3em} LSE_\eps^+ \pth{\revision{\bm F_E}(\revision{\bm z_E}, \revision{\bm \lambda_E})} - \frac{\log(n\revision{_{F_E}})}{\epsilon} \leq 0, \label{eq:convexLSE} \\
    &&& \hspace{-3em} LSE_\eps^+ \pth{\bm F_j(\revision{\bm z_j}, \bm \lambda_j)} - \frac{\log(n\revision{_{F_j}})}{\epsilon} \leq 0. \label{eq:convexLSEnumber2}
\end{align}

\begin{figure}
    \centering
    \includegraphics[width=0.75\columnwidth]{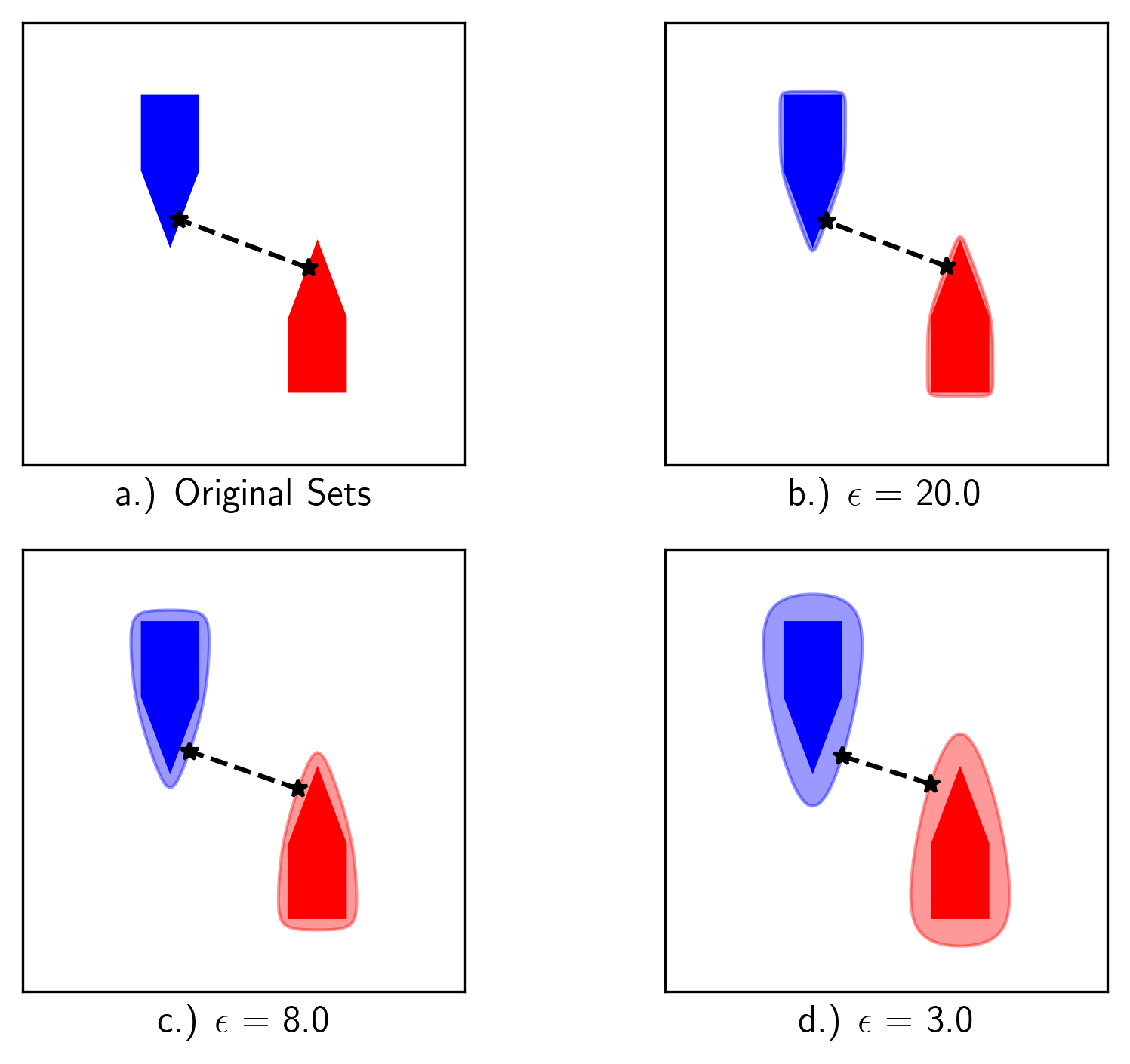}
    \caption{Demonstration of the $LSE_\epsilon^+$ overapproximation of an agent polytope set and the results of the nearest point optimization for various levels of approximation strictness. The nearest points found for each plot is shown as black stars on the boundary of the set with a dashed line showing the connecting vector. a.) A plot of the original polytope sets without any overapproximation. b.) A plot of the polytopes with a very tight overapproximation ($\epsilon = 20$). c.) A plot of the polytopes with a moderate overapproximation ($\epsilon = 8$). d.) A plot of the polytopes with a loose overapproximation ($\epsilon = 3$).}
    \label{fig:near_point_comparison}
\end{figure}

Figure \ref{fig:near_point_comparison} compares two original sets with their \(S_\epsilon^+\) approximations for various values of \(\epsilon\).
Note that by the properties of the \(LSE_\epsilon^+\) function, the degree of overapproximation of the sets \((S_E)_\epsilon^+(\revision{\bm \lambda_E}), \revision{(\overline{S}_j)_\epsilon^+(\bm \lambda_j)}\) can be made arbitrarily small as \(\epsilon \rightarrow \infty\).
The next subsections will 1) prove that the reformulated optimization problem in \eqref{eq:LSE_min_dist_func} is convex and has unique optimal points \(\revision{\bm z_E^*}, \revision{\bm z_j^*}\), and 2) establish conditions under which the function \(d_\epsilon^+\) is differentiable with respect to \(\revision{\bm \lambda_E}, \bm \lambda_j\).
This proposed method will be shown to allow application of set-valued CBF methods to non-strictly convex \(S_E(\revision{\bm \lambda_E}), \overline{S}_j(\bm \lambda_j)\) sets (such as polytopes) without requiring nonsmooth techniques.

\subsection{Convexity and Unique Optimal Points}

We first demonstrate that \eqref{eq:LSE_min_dist_func} is a convex optimization problem. This will be proven using tools from disciplined convex programming \cite{grant2006disciplined, boyd2004convex}.

\begin{lemma}
    \label{lem:optimization_problem_is_convex}
    The optimization problem in \eqref{eq:LSE_min_dist_func} is convex for any fixed \(\revision{\bm \lambda_E}, \bm \lambda_j\).
\end{lemma}

The next result proves that the optimization problem in \eqref{eq:LSE_min_dist_func} has unique optimal points \(\revision{\bm z_E^*},\revision{\bm z_j^*}\).

\begin{lemma}
    \label{lem:unique_optimal_point}
    Suppose \(S_E(\revision{\bm \lambda_E}) \cap \overline{S}_j(\bm \lambda_j) = \emptyset\). Then any optimal point \((\revision{\bm z_E^*}, \revision{\bm z_j^*})\) to the optimization problem in \eqref{eq:LSE_min_dist_func} is unique.
\end{lemma}

\subsection{Differentiability of Distance Function}
\label{sec:differentiability}

We next demonstrate that the function \(d_\epsilon^+\) is differentiable with respect to the parameters \(\revision{\bm \lambda_E}\) and \(\bm \lambda_j\).
The approach below is derived from recent works on differentiability of convex optimization problems and cone programs \cite{barratt2018differentiability, agrawal2019differentiating, agrawal2019differentiable}.

Consider a convex optimization problem \(\text{min}_{\bm z} \mathcal{F}_0(\bm z, \bm \theta) \text{ s.t. } \mathcal{F}(\bm z, \bm \theta) \leq 0\).
The optimal value is defined as \(p^*(\bm \theta) \triangleq \inf\{\mathcal{F}_0(\bm z, \bm \theta) \mid \mathcal{F}(\bm z, \bm \theta) \leq 0\}\).
The set-valued solution mapping is defined as  \(\mathcal{S}(\bm \theta) \triangleq \{\bm z \mid \mathcal{F}(\bm z, \bm \theta) \leq 0,\ \mathcal{F}_0(\bm z, \bm \theta) = p^*(\bm \theta)\}\).
The Lagrangian is defined as 
\begin{align}
    L(\bm z, \bm \mu, \bm \theta) \triangleq \mathcal{F}_0(\bm x, \bm \theta) + \bm \mu^\intercal \mathcal{F}(\bm z, \bm \theta).
\end{align}
If Slater's condition holds, a point \(\bm z^*\) is optimal if and only if there exists \(\bm \mu^*\) such that the KKT conditions are satisfied:
\begin{align*}
    \mathcal{F}(\bm z^*, \bm \theta) &\leq 0, & \bm \mu^*_i \mathcal{F}_i(\bm z^*, \bm \theta) &= 0 \\
    \bm \mu^* &\geq 0 & \frac{\partial}{\partial \bm z} L(\bm z^*, \bm \mu^*, \bm \theta) &= 0.
\end{align*}
Assuming that strict complementarity holds,\footnote{Given a solution \((\bm z^*, \bm \mu^*)\), strict complementarity holds if the set \(\{i \mid \bm \mu^*_i = 0 \text{ and } \mathcal{F}_i(\bm z^*, \bm \theta) = 0 \} = \emptyset\). \cite{barratt2018differentiability}} the inequality \(\mathcal{F}(\bm z^*, \bm \theta) \leq 0\) can be ignored and the implicit function theorem can be applied to the problem. Optimal solutions \(\bm z^*, \bm \mu^*\) are defined implicitly by the function
\begin{align}
    \label{eq:G^implicit}
    \mathcal{G}(\bm z, \bm \mu, \bm \theta) = \bmx{\frac{\partial}{\partial \bm  z}L(\bm z^*, \bm \mu^* \bm \theta) \\ \textbf{diag}(\bm \mu^*) \mathcal{F}(\bm z, \bm \theta)}
\end{align}
In other words if strict complementarity holds, if Slater's condition also holds, and \revision{Assumption \ref{assume:applicable_sets} is} satisfied, a feasible point \((\bm z^*, \bm \mu^*)\) is optimal if \(\mathcal{G}(\bm z^*, \bm \mu^*, \bm \theta) = 0\).

The implicit function theorem \cite[Thm. 1B.1]{dontchev2009implicit} can be used with equation \eqref{eq:G^implicit} to obtain the gradients \(\frac{\partial \bm d_\epsilon^+}{\partial \revision{\bm \lambda_E}}\). Let \(\bm \nu = \bmx{\bm z^\intercal & \bm \mu^\intercal}^\intercal\) and define the following partial Jacobians:

\begin{align}
    \frac{\partial \mathcal{G}(\bm \nu, \bm \theta)}{\partial \bm \nu} &\triangleq \bmx{
        \frac{\partial}{\partial \bm z}\pth{\frac{\partial}{\partial \bm z} L(\bm \nu, \bm \theta)} & \pth{\frac{\partial}{\partial \bm z} \mathcal{F}(\bm z, \bm \theta)}^\intercal \\
        \textbf{diag}\pth{\bm \mu} \frac{\partial}{\partial \bm z} \mathcal{F}(\bm z, \bm \theta) & \textbf{diag}\pth{\mathcal{F}(\bm z, \bm \theta)}.
    } \label{eq:Jacobian1} \\
    \frac{\partial \mathcal{G}(\bm \nu, \bm \theta)}{\partial \bm \theta} &\triangleq \bmx{
        \frac{\partial}{\partial \bm \theta}\pth{\frac{\partial}{\partial \bm z} L(\bm \nu, \bm \theta)} \\
        \textbf{diag}(\bm \mu)\frac{\partial}{\partial \bm \theta} \mathcal{F}(\bm z, \bm \theta) 
    } \label{eq:Jacobian2}
\end{align}

\begin{theorem}[\cite{barratt2018differentiability}]
    \label{thm:differentiable}
    Suppose that \revision{Assumption \ref{assume:applicable_sets} holds}, Slater's condition and the strict complementary condition are satisfied, and \(\mathcal{G}(\bm \nu^*, \bm \theta) = 0\). If the matrix \(\frac{\partial \mathcal{G}(\bm \nu, \bm \theta)}{\partial \bm \nu}\) is nonsingular, then the solution mapping \(\mathcal{S}(\cdot)\) has a single-valued localization \(\bm \nu(\cdot)\) around \(\bm \nu^*\) that is continuously differentiable in a neighborhood Q of \(\bm \theta\) with Jacobian satisfying
    \begin{align}
        \frac{\partial}{\partial \bm \theta}\bm \nu(\bm \theta) = -\pth{\frac{\partial \mathcal{G}(\bm \nu^*, \bm \theta)}{\partial \bm \nu}}^{-1} \frac{\partial \mathcal{G}(\bm \nu^*, \bm \theta))}{\partial \bm \theta}\ \forall \bm \theta \in Q.
        \label{eq:barratt_equation}
    \end{align}
\end{theorem}

Theorem \ref{thm:differentiable} can be applied to the function \(d_\epsilon^+\) by defining \(\bm \nu \triangleq \bmx{(\revision{\bm z_E})^\intercal & (\revision{\bm z_j})^\intercal & \bm \mu}\) and \(\bm \theta \triangleq \bmx{\revision{\bm \lambda_E}^\intercal & \bm \lambda_j^\intercal}^\intercal\), where \revision{\(\bm \mu \triangleq \bmx{\mu_1 & \mu_2}\) is the vector of Lagrange multipliers for inequality constraints \eqref{eq:convexLSE}-\eqref{eq:convexLSEnumber2}}.
However, Theorem \ref{thm:differentiable} assumes rather than asserts that the Jacobian \(\pdrv{\mathcal{G}(\bm \nu^*, \bm \theta)}{\bm \nu}\) is invertible. The natural question that follows is: under what conditions is this Jacobian invertible for problem \eqref{eq:LSE_min_dist_func}?
We now prove that under Assumptions \revision{\ref{assume:applicable_sets} and} \ref{assume:full_column_rank}, the strict complementarity condition holds, the Jacobian is invertible, and the function \(d_\eps^+\) is differentiable with respect to \(\revision{\bm \lambda_E}, \bm \lambda_j\).
First, we present conditions under which strict complementarity holds for the optimization problem \eqref{eq:LSE_min_dist_func}.

\begin{lemma}
    \label{lem:strict_complementarity_holds}
    If \(S_E(\revision{\bm \lambda_E}) \cap \overline{S}_j(\bm \lambda_j) = \emptyset\), then strict complementarity holds for problem \eqref{eq:LSE_min_dist_func}.
\end{lemma}

Next, we establish the invertibility of the Jacobian \(\pdrv{\mathcal{G}(\bm \nu^*, \bm \theta))}{\bm \nu}\) for the optimization problem \eqref{eq:LSE_min_dist_func}.

\begin{theorem}
    Under Assumptions \revision{\ref{assume:applicable_sets} and} \ref{assume:full_column_rank}, if \(S_E(\revision{\bm \lambda_E}) \cap \overline{S}_j(\bm \lambda_j) = \emptyset\) then the Jacobian \(\pdrv{\mathcal{G}(\bm \nu^*, \bm \theta)}{\bm \nu}\) is invertible.
    \label{thm:jacobian_inverse}
\end{theorem}

\begin{proof}
    For notational brevity we define the following matrices:
    \begin{align}
        \begin{aligned}
        &A \triangleq \frac{\partial^2}{\partial \bm z^2} L(\bm z^*, \bm \mu^*, \bm \theta) \\
        &= \frac{\partial^2}{\partial \bm z^2} \bkt{\nrm{\revision{\bm z_E^*} - \revision{\bm z_j^*}}_2^2} +
         \frac{\partial^2}{\partial \bm z^2} \bkt{\revision{\mu^*_1} LSE_\eps^+ \pth{\revision{\bm F_E}(\revision{\bm z_E^*}, \revision{\bm \lambda_E})}} \\
        & \hspace{1em} + \frac{\partial^2}{\partial \bm z^2} \bkt{\revision{\mu^*_2} LSE_\eps^+ \pth{\bm F_j(\revision{\bm z_j^*}, \bm \lambda_j)}}, \\
        &= \bmx{I & -I \\ -I & I} + \\
        &\bmxs{\frac{\partial^2}{(\partial \bm z_E)^2} \pth{\revision{\mu^*_1} LSE_\eps^+ \pth{\revision{\bm F_E}(\revision{\bm z_E^*}, \revision{\bm \lambda_E})}} & 0 \\ 0 & \frac{\partial^2}{(\partial \bm z_j)^2} \bkt{\revision{\mu^*_2} LSE_\eps^+ \pth{\bm F_j(\revision{\bm z_j^*}, \bm \lambda_j)}}} \label{eq:Amatrix}
        \end{aligned}
    \end{align}
    \vspace{-1em}
    \begin{align}
            B &\triangleq \bmx{\pdrv{}{\bm z} LSE_\eps^+ \pth{\revision{\bm F_E}(\revision{\bm z_E^*}, \revision{\bm \lambda_E})} \\ \pdrv{}{\bm z} LSE_\eps^+ \pth{\bm F_j(\revision{\bm z_j^*}, \bm \lambda_j)}} \nonumber \\
            &= \bmx{\pdrv{}{\revision{\bm z_E}} LSE_\eps^+ \pth{\revision{\bm F_E}(\revision{\bm z_E^*}, \revision{\bm \lambda_E})} & 0 \\ 0 & \pdrv{}{\revision{\bm z_j}} LSE_\eps^+ \pth{\bm F_j(\revision{\bm z_j^*}, \bm \lambda_j)}} \label{eq:Bmatrix} \\
            C &\triangleq \textbf{diag}(\bm \mu^*) \\
            D &\triangleq \textbf{diag}\pth{
                \bmx{LSE_\eps^+ \pth{\revision{\bm F_E}(\revision{\bm z_E^*}, \revision{\bm \lambda_E})} \\ LSE_\eps^+ \pth{\bm F_j(\revision{\bm z_j^*}, \bm \lambda_j)}}
            }
    \end{align}
    Note that the dimensions of \(A\) and \(B\) are \(A \in \R^{2N \times 2N}\) and \(B \in \R^{2 \times 2N}\).
    Substituting these abbreviations into \eqref{eq:Jacobian1} yields
    \begin{align}
        \pdrv{\mathcal{G}(\bm \nu, \bm \theta)}{\bm \nu} = \bmx{A & B^T \\ \textbf{diag}(\bm \mu^*) B & D}.
    \end{align}
    By Lemma \ref{lem:unique_optimal_point}, both constraints to the optimization problem in \eqref{eq:LSE_min_dist_func} are active, implying that \(D = 0\). We therefore have \(\pdrv{\mathcal{G}(\bm \nu, \bm \theta)}{\bm \nu} = \bmxs{A & B^\intercal \\ \textbf{diag}(\bm \mu^*) B & \bm 0_{2\times 2}}.\)
    The inverse of a block diagonal matrix in the form $\bmxs{A & B^\intercal \\ C & \bm 0}$ has the following closed form \cite[Prop. 3.9.7]{bernstein2018scalar}:
    \begin{align*}
        &\bmx{A & B^\intercal \\ C & \bm 0}^{-1} = \\
        &\bmxs{
            A^{-1} + A^{-1}B^\intercal(-CA^{-1}B^\intercal)^{-1}CA^{-1} & -A^{-1}B^\intercal(-CA^{-1}B^\intercal)^{-1} \\
            (CA^{-1}B^\intercal)^{-1}CA^{-1} & (-CA^{-1}B^\intercal)^{-1}
        }
    \end{align*}
    It follows that \(\pdrv{\mathcal{G}(\bm \nu^*, \bm \theta)}{\bm \nu}\) is invertible if and only if \(A\) and \(CA^{-1}B^\intercal = \textbf{diag}(\bm \mu^*) B A^{-1} B^\intercal\) are invertible.
    We first consider \(A\).
    By \eqref{eq:Amatrix}, \(A\) is the sum of the Hessians of the convex objective function and a block diagonal matrix containing the Hessians of the two \(LSE_\eps^+\) constraint functions.
    A function is convex if and only if its Hessian is positive semidefinite (PSD) \cite{boyd2004convex}. In addition, by Lemma \ref{lem:LSE_positive_def} the Hessians of both \(LSE_\eps^+\) constraints are positive definite (PD), implying that the block diagonal matrix is PD.
    Since the sum of PD and PSD matrices is PD, the \(A\) matrix is therefore PD. This implies that \(A\) has strictly positive eigenvalues and is therefore invertible.

    We next prove that \(\textbf{diag}(\bm \mu^*) B A^{-1} B^\intercal\) is invertible. 
    First, we demonstrate that \(B^\intercal\) has full column rank (equivalently, \(B\) has full row rank). This can be seen through a simple proof by contradiction: if \(B^\intercal \in \R^{2N \times 2}\) was not full column rank, then by \eqref{eq:Bmatrix} both \(\pdrv{}{\bm z} LSE_\eps^+(\revision{\bm F_E}(\revision{\bm z_E^*}, \revision{\bm \lambda_E})) = 0\) and \(\pdrv{}{\bm z} LSE_\eps^+(\bm F_j(\revision{\bm z_j^*}, \revision{\bm \lambda_j})) = 0\). \revision{By} \eqref{eq:Lagrangian_f_gradient_term} this would imply that \(\bmxs{I & -I \\ -I & I}\bmxs{\revision{\bm z_E^*} \\ \revision{\bm z_j^*}} = 0\)\revision{, which implies \(\revision{\bm z_E^*} = \revision{\bm z_j^*}\) and} contradicts our assumption that \(S_E(\revision{\bm \lambda_E}) \cap \overline{S}_j(\bm \lambda_j) = \emptyset\). Therefore, \(B^\intercal\) must be full column rank \revision{and \(B\) must be full row rank.}
    
    Observe that \(BA^{-1}B^\intercal\) is a PD matrix. This can be seen by observing that it is symmetric, \(B\) has full row rank, and for any \(\bm y \in \R^{2},\ \bm y \neq 0\) we have \(\bm y^\intercal B A^{-1} B^\intercal \bm y = (\bm y^\intercal B)A^{-1}(B^\intercal \bm y) > 0\).\footnote{Note that \(\bm y \neq 0 \implies B^\intercal \bm y \neq 0\) since \(B\) is full row rank and therefore \(B^\intercal\) is full column rank.} This implies that the eigenvalues of \(B A^{-1} B^\intercal\) are strictly positive. 
    Next, since strict complementarity holds by Lemma \ref{lem:strict_complementarity_holds}, the matrix \(\textbf{diag}(\bm \mu^*)\) is PD since it is symmetric and has strictly positive eigenvalues. The product \(\textbf{diag}(\bm \mu^*) B A^{-1} B^\intercal\) therefore has strictly positive and real eigenvalues since it is the product of two PD matrices. This can be seen as follows:
    Consider two PD matrices \(\mathcal{A},\mathcal{B}\) and observe that \(\mathcal{A}\mathcal{B} = \mathcal{A}\mathcal{B}^{1/2}\mathcal{B}^{1/2} = \mathcal{B}^{-1/2}\pth{\mathcal{B}^{1/2}\mathcal{A}\mathcal{B}^{1/2}}\mathcal{B}^{1/2}.\)
    This implies that the matrix \(\mathcal{A}\mathcal{B}\) is similar to the PD matrix \(\pth{\mathcal{B}^{1/2}\mathcal{A}\mathcal{B}^{1/2}}\) and shares all (strictly positive) eigenvalues. It follows that \(\textbf{diag}(\bm \mu^*) B A^{-1} B^\intercal\) has strictly positive eigenvalues and is therefore invertible.

    Since \(A^{-1}\) and \(\pth{\textbf{diag}(\bm \mu^*) B A^{-1} B^\intercal}^{-1} = \pth{CA^{-1}B^\intercal}^{-1}\) exist, the Jacobian \(\pdrv{\mathcal{G}(\bm \nu^*, \bm \theta)}{\bm \nu}\) is therefore invertible, which concludes the proof.
\end{proof}

\begin{corr}
    Under Assumptions \revision{\ref{assume:applicable_sets} and} \ref{assume:full_column_rank}, if \(S_E(\revision{\bm \lambda_E}) \cap \overline{S}_j(\bm \lambda_j) = \emptyset\) then \(d_\epsilon^+\) is differentiable with respect to \(\bm \theta = \bmx{\revision{\bm \lambda_E}^\intercal, \bm \lambda_j^\intercal}^\intercal\).
\end{corr}

\begin{proof}
    Observe that \(d_\epsilon^+(\bm \theta) = d_\epsilon^+(\revision{\bm \lambda_E}, \bm \lambda_j) = \nrm{\bmx{1 & -1 & 0} \bm \nu^*(\bm \theta)}_2^2\), implying that \(\pdrv{d_\eps^+}{\bm \theta} = \pdrv{d_\epsilon^+}{\bm \nu^*} \pdrv{\bm \nu^*(\bm \theta)}{\bm \theta}\).
    By Theorem \ref{thm:jacobian_inverse}, the Jacobian \(\pdrv{\mathcal{G}(\bm \nu, \bm \theta)}{\bm \nu}\) is invertible. It follows from Theorem \ref{thm:differentiable} that \(\pdrv{\bm \nu^*(\bm \theta)}{\bm \theta}\) exists.
\end{proof}

\subsection{Computation of Safety-Preserving Control Inputs Using Convex Optimization}
\label{sec:computing_safe_input}

We are now ready to derive a solution to Problem \ref{prob:computecontrolinputs}.
The objective of ensuring \(S_E(\revision{\bm \lambda_E}) \subset S(\bm \lambda)\ \forall t \geq t_0\) is equivalent to ensuring that the ego set never intersects with the unsafe set; i.e., \(S_E(\revision{\bm \lambda_E}) \cap \overline{S}(\bm \lambda) = \emptyset\ \forall t \geq t_0\). By Assumption 1, a sufficient condition for this is that the ego set never intersects with any of the unsafe subsets \(\overline{S}_j(\bm \lambda_j)\); i.e., \(S_E(\revision{\bm \lambda_E}) \cap \overline{S}_j(\bm \lambda_j) = \emptyset\ \forall j \in \{1,\ldots J\},\ \forall t \geq t_0\).

Let \(R > 0\) be a specified safety distance and consider the following family of \revision{functions describing each unsafe subset indexed by $j$}:
\begin{align}
    h_j(\revision{\bm \lambda_E(t)}, \bm \lambda_j(t)) \triangleq d_\epsilon^+(\revision{\bm \lambda_E(t)}, \bm \lambda_j(t)) - R.
    \label{eq:h_j_definition}
\end{align}
It follows that \(h_j(\revision{\bm \lambda_E}, \bm \lambda_j) \geq 0\) implies safety in the sense that \(S_E(\revision{\bm \lambda_E}) \cap \overline{S}_j(\bm \lambda_j) = \emptyset\).
The following well-known Lemma establishes conditions on \(\dot{h}_j\) to guarantee that \revision{\(h_j(\revision{\bm \lambda_E}, \bm \lambda_j) \geq 0\)} for forward time \(t \geq t_0\).

\begin{lemma}[\cite{glotfelter2017nonsmooth}, Lemma 2 ]
    \label{lem:hdot_geq_0}
    Let \(\alpha: \R \rarr \R\) be a locally Lipschitz, extended class-\(\mathcal{K}\) function and \(h: [t_0, t_1] \rightarrow \R\) be an absolutely continuous function. If \(\dot{h}(t) \geq -\alpha(h(t))\) for almost every \(t \in [t_0, t_1]\), and if \(h(0) \geq 0\), then there exists a class-\(\mathcal{KL}\) function \(\beta: \R_{\geq 0} \times \R_{\geq 0} \rarr \R_{\geq 0}\) such that \(h(t) \geq \beta(h(0), t)\), and \(h(t) \geq 0\) for all \(t \in [t_0, t_1]\).
\end{lemma}

In light of Lemma \ref{lem:hdot_geq_0}, we seek to ensure that \(\dot{h}_j(t) \geq -\alpha(h_j(t))\) for some class-\(\mathcal{K}\) function \(\alpha\) and for all \(t \geq t_0\). Expanding out the expression for \(\dot{h}_j\) and recalling that \(\revision{\dot{\bm \lambda}_E} = \bm f(\revision{\bm \lambda_E}) + \bm g(\revision{\bm \lambda_E}) \bm u\) yields the condition
\begin{align}
    \label{eq:hdot_expanded}
    \pdrv{h_j}{\revision{\bm \lambda_E}}\pth{\bm f(\revision{\bm \lambda_E}) + \bm g(\revision{\bm \lambda_E}) \bm u} + \pdrv{h_j}{\bm \lambda_j} \dot{\bm \lambda}_j \geq -\alpha(h_j(\revision{\bm \lambda_E}, \bm \lambda_j)).
\end{align}
The gradients \(\pdrv{h_j}{\revision{\bm \lambda_E}}, \pdrv{h_j}{\bm \lambda_j}\) can be computed using the methods described in Section \ref{sec:differentiability}. \revision{The quantity $\dot{\bm \lambda}_j$ can be estimated using appropriate techniques from the literature.} Since the left hand side of \eqref{eq:hdot_expanded} is affine in \(\bm u\), it can be used as a constraint in classical CBF QP formulations. For example, the following QP minimally modifies a nominal control input \(\bm u_\text{nom}\) to produce a safety-preserving control input \(\bm u_j^*(\revision{\bm \lambda_E}, \bm \lambda_j)\):
\begin{align*}
        &\bm u_j^*(\revision{\bm \lambda_E}, \bm \lambda_j) = \min_{\bm u} \nrm{\bm u - \bm  u_\text{nom}}_2^2 \\
        &\text{s.t.  }  \pdrv{h_j}{\revision{\bm \lambda_E}}\pth{\bm f(\revision{\bm \lambda_E}) + \bm g(\revision{\bm \lambda_E}) \bm u} + \pdrv{h_j}{\bm \lambda_j} \dot{\bm \lambda}_j \geq -\alpha(h_j(\revision{\bm \lambda_E}, \bm \lambda_j)). 
\end{align*}
In addition, \eqref{eq:hdot_expanded} can also be used in combination with a control Lyapunov function constraint to create a QP that computes a simultaneously safe and asymptotically stabilizing control input using a CLF-CBF formulation \cite{ames2019control}.

As per Assumption \ref{assume:unsafe}, the unsafe set \(\overline{S}(\bm \lambda)\) is a subset of the union of a finite number of sets \(\overline{S}_j(\bm \lambda_j)\), \(j=1,\ldots,J\). To compute a control input \(\bm u^*(\revision{\bm \lambda_E}, \bm \lambda)\) that ensures the ego set never intersects with any of these unsafe sets \(\overline{S}_j(\bm \lambda)\), the following two-step process must be executed at each time the control input is updated. First, the convex optimization problem \eqref{eq:LSE_min_dist_func} must be solved for each \(j\) to obtain values of \(h_j(\revision{\bm \lambda_E}, \bm \lambda_j), \pdrv{h_j}{\revision{\bm \lambda_E}}, \pdrv{h_j}{\bm \lambda_j}\) for each \(j\) as per \eqref{eq:h_j_definition}. Second, the following QP must be solved to obtain \(\bm u^*(\revision{\bm \lambda_E}, \bm \lambda)\):
\begin{align}
    \begin{aligned}
        &\bm u^*(\revision{\bm \lambda_E}, \bm \lambda) = \min_{\bm u} \nrm{\bm u - \bm  u_\text{nom}}_2^2 \\
        &\text{s.t.  }  \pdrv{h_1}{\revision{\bm \lambda_E}}\pth{\bm f(\revision{\bm \lambda_E}) + \bm g(\revision{\bm \lambda_E}) \bm u} + \pdrv{h_1}{\bm \lambda_1} \dot{\bm \lambda}_1 \geq -\alpha_1(h_1(\revision{\bm \lambda_E}, \bm \lambda_1)), \\
        &\hspace{2em}\vdots \\
        &\hspace{1em}\pdrv{h_J}{\revision{\bm \lambda_E}}\pth{\bm f(\revision{\bm \lambda_E}) + \bm g(\revision{\bm \lambda_E}) \bm u} + \pdrv{h_J}{\bm \lambda_J} \dot{\bm \lambda}_J \geq -\alpha_J(h_J(\revision{\bm \lambda_E}, \bm \lambda_J)).
        \label{eq:cbf_qp_multiconstraint}
    \end{aligned}
\end{align}
Here, \(\alpha_1,\ldots,\alpha_J\) are each locally Lipschitz, extended class-\(\mathcal{K}\) functions.
We define \(K(\revision{\bm \lambda_E}, \bm \lambda)\) as the set of feasible control inputs to \eqref{eq:cbf_qp_multiconstraint}:
\begin{align*}
    &K(\revision{\bm \lambda_E}, \bm \lambda) \triangleq \bigg\{\bm u \mid \\
    &\hspace{1em}\pdrv{h_1}{\revision{\bm \lambda_E}}\pth{\bm f(\revision{\bm \lambda_E}) + \bm g(\revision{\bm \lambda_E}) \bm u} + \pdrv{h_1}{\bm \lambda_1} \dot{\bm \lambda}_1 \geq -\alpha(h_1(\revision{\bm \lambda_E}, \bm \lambda_1)), \\
    &\hspace{2em} \vdots \\
    &\hspace{1em} \pdrv{h_J}{\revision{\bm \lambda_E}}\pth{\bm f(\revision{\bm \lambda_E}) + \bm g(\revision{\bm \lambda_E}) \bm u} + \pdrv{h_J}{\bm \lambda_J} \dot{\bm \lambda}_J \geq -\alpha(h_J(\revision{\bm \lambda_E}, \bm \lambda_J)) \bigg\}
\end{align*}
The following theorem establishes that control inputs \(\bm u(t) \in K(\revision{\bm \lambda_E}, \bm \lambda)\) solve Problem \ref{prob:computecontrolinputs} given in Section \ref{sec:problem_formulation}.

\begin{theorem}
    Let \(S_E(\revision{\bm \lambda_E}(t))\) be the ego set defined and \(\overline{S}(\bm \lambda(t))\) be the unsafe set as defined in Section \ref{sec:problem_formulation}. Let \(\bm u^*(\revision{\bm \lambda_E}(t))\) be any Lipschitz continuous controller such that \(\bm u^*(\revision{\bm \lambda_E}(t)) \in K(\revision{\bm \lambda_E}(t), \bm \lambda(t))\) for all \(t \in [t_0, t_1)\). If \(S_E(\revision{\bm \lambda_E}(t_0)) \cap \overline{S}(\bm \lambda(t_0)) = \emptyset\), then \(S_E(\revision{\bm \lambda_E}(t)) \cap \overline{S}(\bm \lambda(t)) = \emptyset\) for all \(t \in [t_0, t_1]\).
\end{theorem}

\begin{proof}
    By \eqref{eq:cbf_qp_multiconstraint}, we have \(\dot{h}_1 \geq -\alpha_1(h_1(\revision{\bm \lambda_E(t)}, \bm \lambda(t))),\ldots,\dot{h}_J \geq -\alpha_J(h_1)(\revision{\bm \lambda_E(t)}, \bm \lambda(t)).\)
    By Lemma \ref{lem:hdot_geq_0} this implies that \(h_1(t) \geq 0, \ldots, h_J(t) \geq 0\) for all \(t \in [t_0, t_1]\). By \eqref{eq:h_j_definition}, this implies that \(d_\epsilon^+(\revision{\bm \lambda_E(t)}, \bm \lambda_j(t)) \geq R \geq 0\) \(\forall j = 1,\ldots,J\), \(\forall t \in [t_0, t_1]\) which implies that \(S_E(\revision{\bm \lambda_E}(t)) \cap \overline{S}(\bm \lambda(t)) = \emptyset\) for all \(t \in [t_0, t_1]\).
\end{proof}

Algorithm \ref{alg:compute} summarizes the process for computing a safety-preserving control input \(\bm u(\revision{\bm \lambda_E}, \bm \lambda) \in K(\revision{\bm \lambda_E}, \bm \lambda)\).

\begin{remark}
    The requirement that \(\bm u^* \in K(\revision{\bm \lambda_E}, \bm \lambda)\) must be Lipschitz continuous can be relaxed to \(\bm u^*\) simply being Lebesgue measurable using results such as \cite{usevitch2020strong}. However, due to space constraints we do not consider this extension here.
\end{remark}

\begin{algorithm}
\caption{Computing Safety-Preserving Control Inputs}
\label{alg:compute}
\begin{algorithmic}
\State At each control input update time $t$:
\Indent
\State Compute \(\bm u_{\text{nom}}\)
\For{$j=1,\ldots,J$}
    \State \multiline{Solve convex problem \eqref{eq:LSE_min_dist_func}, use \eqref{eq:h_j_definition} to compute \(h_j(\revision{\bm \lambda_E}, \bm \lambda_j)\)}
    \State Compute Jacobians $\pdrv{h_j}{\revision{\bm \lambda_E}},\ \pdrv{h_j}{\bm \lambda_j}$ 
\EndFor
\State Solve convex problem \eqref{eq:cbf_qp_multiconstraint} to compute \(\bm u^*(\revision{\bm \lambda_E}, \bm \lambda)\)
\EndIndent
\end{algorithmic}
\end{algorithm}

\section{Simulations}
\label{sec:simulations}

In this section, we demonstrate the efficacy of our method in computing safety-preserving control inputs for robot agents with the results of our simulations.
\revision{Our demonstrative simulation is run in Python version 3.12.}
\revision{Our code uses the JAX package \cite{jax2018github} for automatic differentiation and solving Quadratic Programs (QP) as well as the cvxpy and cvxpylayers packages \cite{diamond2016cvxpy,cvxpylayers2019} to solve non-QP convex optimization problems.}
Our simulation consists of four robot agents with convex polytope-shaped bodies \revision{traversing the environment} to reach their respective goal states.
Each individual agent $i \in 1,\ldots,4$, is represented as an ego set of points \(S_E^i(\revision{\bm \lambda_E^i}) \subset \R^2\) with a parameter vector \(\revision{\bm \lambda_E^i} = \bmx{x_{c1}^i & x_{c2}^i & \theta^i}^\intercal\), where \(\bm x_c^i = \bmx{x_{c1}^i & x_{c2}^i}^\intercal\) is the Euclidean position of the agent and \(\theta^i\) is the orientation of the agent as an angle \revision{counterclockwise from the positive x-axis in radians}.
The set \(S_E^i(\revision{\bm \lambda_E^i})\) has the form of a rigid polytope, i.e., halfspace inequalities in the form \(A^i(\revision{\bm \lambda_E^i})\bm x + \bm b^i(\revision{\bm \lambda_E^i}) \leq 0\), \(\bm x \in \R^2\), with position and orientation determined by the values of \(\revision{\bm \lambda_E^i}\).

The agents have unicycle dynamics in $\R^2$ with its state being the origin frame of its set.
The unicycles are controlled by input-output linearization \cite{siciliano2009mobile} with each agent having the output \(\bm y^i = \bmx{y_1^i & y_2^i}^\intercal\) defined as \(\bmxs{y_1^i \\ y_2^i} = \bmxs{x_{c1}^i + b\cos{\theta^i} \\ x_{c2}^i + b\sin{\theta^i}},\ b >0\).
The output $\bm y^i$ is treated as having single integrator dynamics \(\dot{\bm y^i} = \bm u^i = \bmx{u_1^i & u_2^i}^\intercal\).
The control inputs for the unicycle dynamics $\bmx{v^i & \omega^i}^\intercal$ are then computed using the transformation \(\bmxs{v^i \\ \omega^i} = T^{-1}(\theta^i)\bm u^i = \bmxs{
        \cos{\theta^i} & \sin{\theta^i} \\
        -\frac{\sin{\theta^i}}{b} & \frac{\cos{\theta^i}}{b}
    }
    \bmxs{u_1^i \\ u_2^i}.\)

During the simulation, each agent is nominally controlled by a go-to-goal control law defined as \(\bm u_{nom}^i = -k_u(\bm x_c^i - \bm x_{c,goal}^i)\); where $k_u > 0$ is a scalar control gain and \(\bm x_{c,goal}^i = \bmx{x_{c1,goal}^i & x_{c2,goal}^i}^\intercal\) is the agent's goal position.
The agents each minimally modify their nominal control input $\bm u_{nom}$ using the functions \(h_j\) as described in algorithm \ref{alg:compute}.
Each agent $i \in 1,\ldots,4$ runs the algorithm with \(\overline{S}^i(\bm \lambda) = \bigcup_{j \neq i} S_E^j(\bm \lambda_j)\) where $\bm \lambda_j = \revision{\bm \lambda_E^j}$.
Since the control input of the input-output linearized dynamics is in terms of $\dot{\bm y}$, we made the following adjustment to the dynamics equations to account for this: \(\revision{\dot{\bm \lambda}_E^i} = \bm f(\revision{\bm \lambda_E^i}) + \tilde{\bm g}(\revision{\bm \lambda_E^i}) \bm u^i,\) \(\tilde{\bm g}(\revision{\bm \lambda_E^i}) = \bm g(\revision{\bm \lambda_E^i})T^{-1}(\theta^i).\)
The matrix $T^{-1}(\theta^i)$ is the transform matrix that converts from $\dot{\bm y}^i$ back to $\bmx{v & \omega}^\intercal$.
This form allows us to use the functions \(h\) to directly modify the linearized inputs to preserve safety.

\begin{figure}
    \centering
    \includegraphics[width=\columnwidth]{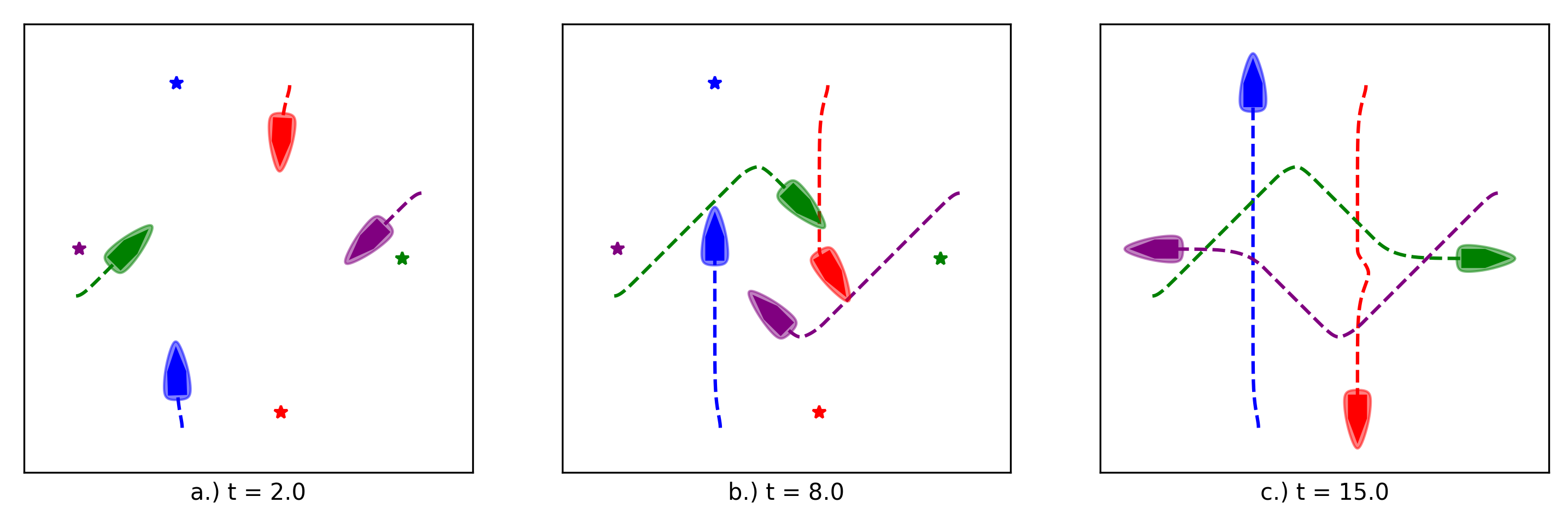}
    \caption{Still frames from our simulation. Agents are represented by their polytope shape with its semi-transparent $LSE_\epsilon^+$ approximation around it. Agents are attempting to reach corresponding goal positions while avoiding collisions. The goal for each agent is shown as a similar colored star. This simulation demonstrates the efficacy of our method preventing collisions between sets for all agents.}
    \label{fig:sim_screenshots}
\end{figure}

\begin{figure}
    \centering
    \includegraphics[width=0.8\columnwidth]{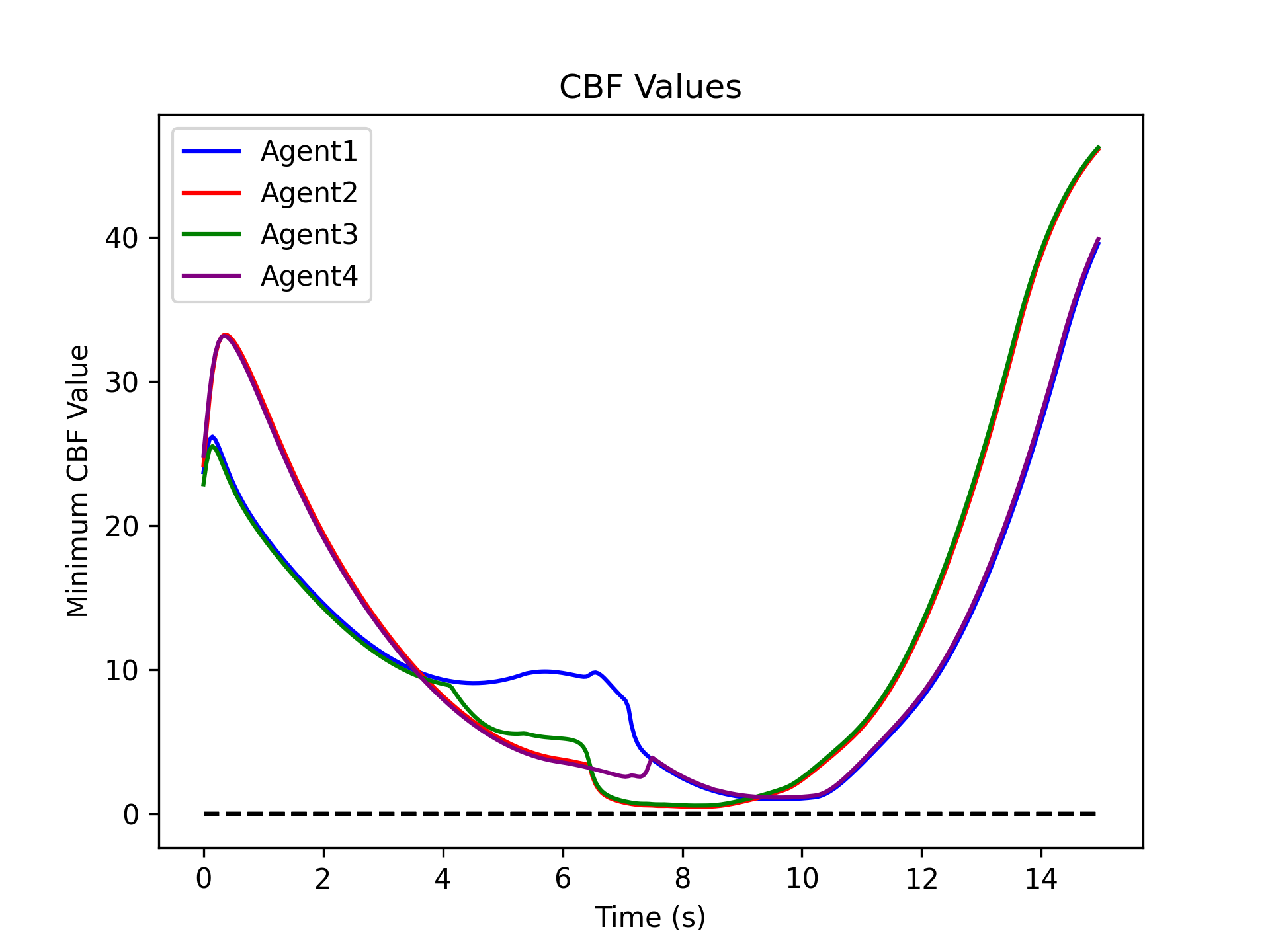}
    \caption{Plot of the minimum \(h\) function value for each agent over time. A separate \(h_j^i\) function is run between each agent $i$ and each external agent $j\neq i$. The value plotted for each agent $i$ is \(\min_j h_j^i(\cdot)\), implying that safety is preserved if the minimum value is greater than 0. As demonstrated by the data, the minimum values were kept above 0 and thus the agents were kept within the safe set throughout the whole simulation.}
    \label{fig:sim_cbfs}
\end{figure}

Still frames from the simulation can be seen in figure \ref{fig:sim_screenshots}.
During the simulation, the agents altered their trajectories by each turning left to avoid collisions.
These coordinated turns resulted in the agents rotating clockwise around the central area until they were each able to proceed to their goals unimpeded.
Throughout the simulation, each agent $i$ ran a \revision{safe set function} as per \eqref{eq:h_j_definition} with each other agent $j$ to calculate the corresponding safety constraints. The \revision{safe set function values} indicate safety when $h_j^i(\revision{\bm \lambda_E^i}, \bm \lambda_j) \geq 0$.
Figure \ref{fig:sim_cbfs} shows the minimum value of the \revision{safe set functions} run by each agent $i$.
This figure demonstrates that since the minimum \revision{safe set function value} of each agent was always greater than zero, none of the agents ever entered the unsafe set.

\section{Conclusion}
\label{sec:conclusion}

This paper presented a novel method to compute control inputs that prevent a parameterized time-varying ego set from intersecting with an unsafe set using convex optimization formulations. The method leveraged a log-sum-exp-based approximation of the squared minimum distance between sets that can be tightened to an arbitrary degree. Future work will investigate extensions of this method to multi-robot systems with high relative degree safety constraints.

\section{Appendix}

\revision{
\begin{proof}[\textbf{Proof of Lemma \ref{lem:LSE_bounds}}]
    From \cite[Eq.4]{nielsen2016guaranteed} we have $\max\pth{\bm x} < LSE(\bm x) \leq \max\pth{\bm x} + \log(q)$. Let $\bm x^+ = (0, x_1, \ldots, x_q)$, then $\max\pth{\bm x^+} < LSE^+(\bm x) \leq \max\pth{\bm x^+} + \log(q)$. Multiplying each $x$ by $\eps$ gives us $\max\pth{\eps \bm x^+} < LSE^+(\eps \bm x) \leq \max\pth{\eps \bm x^+} + \log(q)$ and simplifies to $\max\pth{\bm x^+} < LSE_\eps^+(\bm x) \leq \max\pth{\bm x^+} + \frac{\log(q)}{\eps}$
\end{proof}
}

\begin{proof}[\textbf{Proof of Lemma \ref{lem:LSE_positive_def}}]
    Define \(\tilde{\bm z} \triangleq \bmx{e^{\eps x_1} & \cdots & e^{\eps x_q}}\) \revision{and recall that \(e^0 = 1\)}.
    The gradient of \(LSE_\epsilon^+\) is \revision{\(\pdrv{}{\bm x}LSE_\eps^+ = \pth{\frac{1}{1 + \bm 1^\intercal \tilde{\bm z}}}\tilde{\bm z}.\)}
    The Hessian of \(LSE_\eps^+\) is therefore
    \begin{align*}
        \pdrv{^2}{\bm x^2}LSE_\eps^+
        &= \revision{\frac{\eps}{\pth{1 + \bm 1^\intercal \tilde{\bm z}}^2}\pth{ \textbf{diag}(\tilde{\bm z}) + \bm 1^\intercal \tilde{\bm z}\ \textbf{diag}(\tilde{\bm z}) - \tilde{\bm z}\tilde{\bm z}^\intercal}}
    \end{align*}
    The matrix \(\pth{\bm 1^\intercal z\ \textbf{diag}(\tilde{\bm z}) - \tilde{\bm z}\tilde{\bm z}^\intercal}\) is symmetric and positive semidefinite (PSD) \cite{boyd2004convex}. The matrix \revision{\(\textbf{diag}(\tilde{\bm z})\)} is diagonal with strictly positive entries, and is therefore positive definite (PD). Since the sum of a PD and PSD matrix is PD, the Hessian \(\pdrv{^2}{\bm x^2}LSE_\eps^+\) is PD, implying that \(LSE_\eps^+\) is strictly convex.
\end{proof}

\begin{proof}[\textbf{Proof of Lemma \ref{lem:S_subset_Sepsplus}}]
    Choose any \(\bm \lambda \in \R^M\) and consider any \(\bm x \in S(\revision{\bm \lambda})\). This implies that \(\max\pth{0, F^1(\bm x, \revision{\bm \lambda}), \ldots, F^q(\bm x, \revision{\bm \lambda})} \leq 0\). Using \eqref{eq:LSE_plus_def}, we have \(LSE_\epsilon^+(\bm F(\bm x,\revision{\bm \lambda})) \leq
        \max\pth{0, F^1(\bm x, \revision{\bm \lambda}), \ldots, F^q(\bm x, \revision{\bm \lambda})} + \frac{\log(q)}{\epsilon} \leq 0 + \frac{\log(q)}{\epsilon}.\)
    Therefore \(\bm x \in S_\epsilon^+(\revision{\bm \lambda})\).
\end{proof}


\begin{proof}[\textbf{Proof of Lemma \ref{lemma:assume_full_column_rank}}]
    We prove by contradiction.
    Suppose that \(\pdrv{\revision{\bm F_E}(\bm x, \revision{\bm \lambda_E})}{\bm x}\) is not full rank for some \(\revision{\bm \lambda_E}\). 
    Since each \(\revision{F_E^i}\) is affine, \(\revision{\bm F_E}(\bm x, \revision{\bm \lambda_E})\) has the form \(A(\revision{\bm \lambda_E})\bm x - b(\revision{\bm \lambda_E})\) and the Jacobian has the form \(\pdrv{\revision{\bm F_E}(\bm x, \revision{\bm \lambda_E})}{\bm x} = A(\revision{\bm \lambda_E}) \in \R^{n\revision{_{F_E}} \times N}\).
    This implies the existence of a vector \(\hat{\bm x} \in \R^N\) such that \(A(\revision{\bm \lambda_E}) \hat{\bm x} = \bm 0\).

    Choose any feasible \(\bm x_0\) such that \(\revision{\bm F_E}(\bm x_0, \revision{\bm \lambda_E}) = A(\revision{\bm \lambda_E})\bm x_0 - b(\revision{\bm \lambda_E}) \leq 0\). Consider the cone \(\bm x_0 + \gamma \hat{\bm x}\) with \(\gamma \geq 0\). Then \(\lim_{\gamma \rarr \infty} A(\revision{\bm \lambda_E})(\bm x_0 + \gamma \hat{\bm x}) - b(\revision{\bm \lambda_E}) \leq 0\) while \(\lim_{\gamma \rarr \infty} \nrm{\bm x_0 + \gamma \hat{\bm x}} = \infty\), contradicting the assumption that the set \(\{\bm x: \revision{\bm F_E}(\bm x, \revision{\bm \lambda_E}) \leq 0\}\) is compact. Therefore, \(\pdrv{\revision{\bm F_E}(\bm x, \revision{\bm \lambda_E})}{\bm x}\) must have full column rank. Similar arguments can be applied to the constraints \(\bm F_j(\bm x, \bm \lambda_i)\).
\end{proof}

\begin{proof}[\textbf{Proof of Lemma \ref{lem:optimization_problem_is_convex}}]
    First, note that the function \(LSE_\epsilon^+(\cdot)\) is strictly convex and nondecreasing in each of its arguments.
    Consider the first constraint \(\revision{\bm z_E} \in (S_E)_\epsilon^+(\revision{\bm \lambda_E})\) and recall that by Assumption \revision{\ref{assume:applicable_sets}} we have \(S_E(\revision{\bm \lambda_E}) = \{\bm x: \revision{F_E^i}(\bm x, \revision{\bm \lambda_E}) \leq 0 \}\).
    Consider the first constraint \eqref{eq:convexLSE}.
    Since each \(\revision{F_E^i}\) is convex in \(\revision{\bm z_E}\), it follows that the composition \(LSE_\eps^+ \pth{\revision{F_E}^1(\revision{\bm z_E}, \revision{\bm \lambda_E}), \ldots, \revision{F_E}^{n\revision{_{F_E}}}(\revision{\bm z_E}, \revision{\bm \lambda_E})}\) is convex in \(\revision{\bm z_E}\). It follows that the constraint in \eqref{eq:convexLSE} is a convex constraint.
    Next, consider the second constraint \eqref{eq:convexLSEnumber2}.
    Since each \(F_j^k\) is convex in \(\revision{\bm z_j}\), the composition \(LSE_\eps^+ \pth{\bm F_j(\revision{\bm z_j}, \revision{\bm \lambda_j})}\) is convex in \(\revision{\bm z_j}\). It likewise follows that the constraint in \eqref{eq:convexLSEnumber2} is a convex constraint.
    Since the objective function is convex and the two constraints are convex inequality constraints, the optimization problem in \eqref{eq:LSE_min_dist_func} is convex.
\end{proof}

\begin{proof}[\textbf{Proof of Lemma \ref{lem:unique_optimal_point}}]
    Consider any optimal point \((\revision{\bm z_E^*}, \revision{\bm z_j^*})\). Observe that both of the constraints \eqref{eq:convexLSE},\eqref{eq:convexLSEnumber2} must be active at this point, implying \(LSE_\eps^+ \pth{\revision{\bm F_E}(\revision{\bm z_E}, \revision{\bm \lambda_E})} - \frac{\log(n\revision{_{F_E}})}{\epsilon} = 0\) and \(LSE_\eps^+ \pth{\bm F_j(\revision{\bm z_j}, \revision{\bm \lambda_j})} - \frac{\log(n\revision{_{F_j}})}{\epsilon} = 0\).
    This can be seen by proof through contradiction. Without loss of generality, suppose that the first constraint is inactive at the optimal point \((\revision{\bm z_E^*}, \revision{\bm z_j^*})\); i.e., \(LSE_\eps^+ \pth{\revision{\bm F_E}(\revision{\bm z_E^*}, \revision{\bm \lambda_E})} - \frac{\log(n\revision{_{F_E}})}{\epsilon} < 0\). Then there exists a unit ball \(B_1 \triangleq B(\revision{\bm z_E^*}, \gamma)\), \(\gamma > 0\) such that \(LSE_\eps^+ \pth{\revision{\bm F_E}(\bm y, \revision{\bm \lambda_E})} - \frac{\log(n\revision{_{F_E}})}{\epsilon} \leq 0 \) for all \(\bm y \in B_1\). Let \(\bm y^*\) be the projection of \(\revision{\bm z_j^*}\) onto \(B_1\). Then it follows that \(\bm y^*\) is a feasible point and \(\nrm{\bm y^* - \revision{\bm z_j^*}}_2^2 < \nrm{\revision{\bm z_E^*} - \revision{\bm z_j^*}}_2^2\). However, this contradicts the assumption that \((\revision{\bm z_E^*}, \revision{\bm z_j^*})\) is an optimal point; therefore the first constraint must be active. Similar arguments can be used to demonstrate that the second constraint \(LSE_\eps^+(\bm F_j(\revision{\bm z_j^*}, \revision{\bm \lambda_j}))\) must likewise be active at any optimal point \((\revision{\bm z_E^*}, \revision{\bm z_j^*})\).

    Next, we prove by contradiction that the optimal point \((\revision{\bm z_E^*}, \revision{\bm z_j^*})\) is unique. Suppose that there exists a second optimal point \((\revision{\overline{\bm z}_E^*}, \revision{\overline{\bm z}_j^*})\). Due to the convexity of the problem, any point \((\theta \revision{\bm z_E^*} + (1-\theta) \revision{\overline{\bm z}_E^*}, \theta \revision{\bm z_j^*} + (1-\theta) \revision{\overline{\bm z}_j^*})\) where \(0 \leq \theta \leq 1\) must also be an optimal point.
    
    Consider any such optimal point with \(0 < \theta < 1\).
    By Corollary \ref{corr:LSE_composition_pd_Hessian}, the composition \(LSE_\eps^+(\revision{\bm F_E}(\bm z, \revision{\bm \lambda_E}))\) is strictly convex in \(\bm z\).
    \revision{Thus the following holds for the first constraint:}
    \begin{align*}
        &LSE_\epsilon^+\pth{\revision{\bm F_E}\pth{\theta \revision{\bm z_E^*} + (1-\theta) \revision{\overline{\bm z}_E^*}, \revision{\bm \lambda_E}}} < \\
        &\quad \theta LSE_\epsilon^+\pth{\revision{\bm F_E}(\revision{\bm z_E^*}, \revision{\bm \lambda_E})} + (1-\theta) LSE_\epsilon^+\pth{\revision{\bm F_E}(\revision{\overline{\bm z}_E^*}, \revision{\bm \lambda_E})} = 0
    \end{align*}
    Similar arguments can be used to show that, for the second constraint,
    \(LSE_\epsilon^+(\bm F_j(\theta \revision{\bm z_j^*} + (1-\theta) \revision{\overline{\bm z}_j^*}, \revision{\bm \lambda_j})) < 0\).
    This implies that \revision{neither constraint} is active for the point \((\theta \revision{\bm z_E^*} + (1-\theta) \revision{\overline{\bm z}_E^*}, \theta \revision{\bm z_j^*} + (1-\theta) \revision{\overline{\bm z}_j^*})\) (where \(0 < \theta < 1\)).
    However this contradicts the previously-shown fact that both constraints must be active at any optimal point. Therefore we conclude that an optimal point \((\revision{\bm z_E^*}, \revision{\bm z_j^*})\) must be unique.
\end{proof}

\begin{proof}[\textbf{Proof of Lemma \ref{lem:strict_complementarity_holds}}]
    We prove by contradiction. The two constraints are \eqref{eq:convexLSE} and \eqref{eq:convexLSEnumber2}, the optimal points are \((\revision{\bm z_E^*}, \revision{\bm z_j^*})\) and the optimal Lagrange multipliers are \(\bm \mu^* = \bmx{\mu_1^*, \mu_2^*}\).
    
    Suppose that \(\bm \mu^* = 0\).
    By the convexity of problem \eqref{eq:LSE_min_dist_func} and Slater's condition holding by \revision{Assumption \ref{assume:applicable_sets}}, the KKT conditions are necessary and sufficient for optimality. The following KKT condition therefore holds: \(\pdrv{f}{\bm z} + \mu_1^* \pdrv{}{\bm z}LSE_\eps^+(\revision{\bm F_E}(\revision{\bm z_E^*}, \revision{\bm \lambda_E})) + \mu_2^* \pdrv{}{\bm z}LSE_\eps^+(\bm F(\revision{\bm z_j^*}, \bm \lambda_j)) = 0.\)
    Expanding the entries and rearranging yields
    \begin{align}
    \begin{aligned}
        \bmx{-\mu_1^* \pth{\pdrv{}{\revision{\bm z_E}}LSE_\eps^+(\revision{\bm F_E}(\revision{\bm z_E^*}, \revision{\bm \lambda_E}))}^\intercal \\ -\mu_2^* \pth{\pdrv{}{\revision{\bm z_j}}LSE_\eps^+(\bm F(\revision{\bm z_j^*}, \bm \lambda_j))}^\intercal} = \bmx{I & -I \\ -I & I}\bmx{\revision{\bm z_E^*} \\ \revision{\bm z_j^*}}. \label{eq:Lagrangian_f_gradient_term}
    \end{aligned}
    \end{align}
    Since \(\bm \mu^* = 0\) by assumption, we then have \(\bmxs{I & -I \\ -I & I}\bmxs{\revision{\bm z_E^*} \\ \revision{\bm z_j^*}} = 0.\)
    This implies that \(\revision{\bm z_E^*} = \revision{\bm z_j^*}\) \revision{, which} contradicts our assumption that \(S(\revision{\bm \lambda_E}) \cap \overline{S}_j(\bm \lambda_j) = \emptyset\). By these arguments, neither \(\mu_1^*\) nor \(\mu_2^*\) can equal zero, and \revision{thus} \(\mu_1^* > 0\) and \(\mu_2^* > 0\).\footnote{Recall that the KKT conditions constrain \(\bm \mu \geq 0\).} By Lemma \ref{lem:unique_optimal_point} both constraints \eqref{eq:convexLSE},\eqref{eq:convexLSEnumber2} are both active at the optimal point, implying that strict complementarity holds.
\end{proof}


\begin{lemma}
    \label{lem:composition_pd_hessian}
    Let \(\revision{\tilde{f}}: \R^m \rarr \R\) be a function satisfying both of the following properties:
    \begin{itemize}
        \item The gradient \(\pdrv{\revision{\tilde{f}}}{\bm y}\) has nonnegative entries for all \(\bm y\); i.e. \(\pdrv{\revision{\tilde{f}}}{y_i} \geq 0\ \forall i=1,\ldots,m\).
        \item The Hessian \(\pdrv{^2 \revision{\tilde{f}}}{\bm y^2}\) is positive definite for all \(\bm y\).
    \end{itemize}
    Let \(\revision{\tilde{\bm g}}: \R^n \rarr \R^m\) be a vector-valued function such that each component \(\revision{\tilde{g}_i}: \R^n \rarr \R\) is convex. 
    If the matrix \(\pdrv{\revision{\tilde{\bm g}}}{\bm z}\) has full column rank, then the Hessian \(\frac{\partial^2}{\partial \bm z^2} \revision{\tilde{f}}(\revision{\tilde{\bm g}}(\bm z))\) is positive definite.
\end{lemma}

\begin{proof}
    The gradient of the composition \(\revision{\tilde{f}}(\revision{\tilde{\bm g}}(\bm z))\) is \(\pdrv{}{\bm z}\revision{\tilde{f}}(\revision{\tilde{\bm g}}(\bm z)) = \pdrv{\revision{\tilde{f}}}{\revision{\tilde{\bm g}}} \pdrv{\revision{\tilde{\bm g}}}{\bm z}\). The Hessian is therefore
    \begin{align}
        \pdrv{}{\bm z} \pth{ \pdrv{\revision{\tilde{f}}}{\revision{\tilde{\bm g}}} \pdrv{\revision{\tilde{\bm g}}}{\bm z}} &= \pth{\sum_{i=1}^m \pdrv{\revision{\tilde{f}}}{\revision{\tilde{g}_i}} \frac{\partial^2 \revision{\tilde{g}_i}}{\partial \bm z^2}} + \pdrv{\revision{\tilde{\bm g}}^\intercal}{\bm z} \pdrv{^2 \revision{\tilde{f}}}{\revision{\tilde{\bm g}^2}} \pdrv{\revision{\tilde{\bm g}}}{\bm z}.
    \end{align}
    By assumption, the Hessian \(\pdrv{^2 \revision{\tilde{f}}}{\revision{\tilde{\bm g}}^2}\) is positive definite (PD) and \(\pdrv{\revision{\tilde{\bm g}}}{\bm z}\) is full rank, implying that the symmetric matrix \(\pdrv{\revision{\tilde{\bm g}}^\intercal}{\bm z} \pdrv{^2 \revision{\tilde{f}}}{\revision{\tilde{\bm g}}^2} \pdrv{\revision{\tilde{\bm g}}}{\bm z}\) is PD. This can be seen by noting for all \(\bm x \neq 0\), \(\bm x^\intercal \pth{\pdrv{\revision{\tilde{\bm g}}^\intercal}{\bm z} \pdrv{^2 \revision{\tilde{f}}}{\revision{\tilde{\bm g}}^2} \pdrv{\revision{\tilde{\bm g}}}{\bm z}} \bm x = \pth{\pdrv{\revision{\tilde{\bm g}}}{\bm z} \bm x}^\intercal \pdrv{^2 \revision{\tilde{f}}}{\revision{\tilde{\bm g}}^2} \pth{\pdrv{\revision{\tilde{\bm g}}}{\bm z} \bm x} > 0\).
    
    Each \(\revision{\tilde{g}_i}\) being convex implies that each Hessian \(\frac{\partial^2 \revision{\tilde{g}_i}}{\partial \bm z^2}\) is positive semidefinite (PSD) \cite[Sec. 3.1.4]{boyd2004convex}. Since each \(\pdrv{\revision{\tilde{f}}}{y_i} \geq 0\), it holds that \(\sum_{i=1}^m \pdrv{f}{\revision{\tilde{g}_i}} \frac{\partial^2 \revision{\tilde{g}_i}}{\partial \bm z^2}\) is the sum of PSD matrices.

    The sum of PSD and PD matrices is PD, which implies that \(\frac{\partial^2}{\partial \bm z)^2} \revision{\tilde{f}}(\revision{\tilde{\bm g}}(\bm z)) = \pdrv{}{\bm z} \pth{ \pdrv{\revision{\tilde{f}}}{\revision{\tilde{\bm g}}} \pdrv{\revision{\tilde{\bm g}}}{\bm z}}\) is therefore PD.
\end{proof}

\begin{corr}
    \label{corr:LSE_composition_pd_Hessian}
    The composition \(LSE_\eps^+\pth{\revision{\tilde{g}_1}(\bm z), \ldots, \revision{\tilde{g}_m}(\bm z)}\) for convex \(\revision{\tilde{g}_i}\) has a positive definite Hessian if \(\pdrv{\revision{\tilde{\bm g}}}{\bm z}\) is full rank.
\end{corr}

\begin{proof}
    Define \(\revision{\tilde{f}} \triangleq LSE_\eps^+\). By the proof of Lemma \ref{lem:LSE_positive_def}, \(\revision{\tilde{f}}\) has nonnegative gradient entries and a positive definite Hessian. The result follows from Lemma \ref{lem:composition_pd_hessian}.
\end{proof}

\bibliographystyle{IEEEtran}
\bibliography{bibliography}

\addtolength{\textheight}{-12cm}   

\end{document}